\documentclass[12pt]{amsart}
\usepackage{epsfig,color}
\usepackage{blindtext}
\usepackage{hyperref}
\usepackage{graphicx}
\usepackage{enumitem} 
\usepackage{url}
\usepackage{amssymb}
\usepackage{graphicx,import}
\usepackage{comment}
\usepackage{esint}

\theoremstyle{definition}

\newtheorem{theorem}{Theorem}[section]

\newtheorem{conjecture}[theorem]{Conjecture}
\newtheorem{proposition}[theorem]{Proposition}
\newtheorem{lemma}[theorem]{Lemma}
\newtheorem{remark}[theorem]{Remark}
\newtheorem{corollary}[theorem]{Corollary}

\newtheorem*{remark*}{Remark}

\numberwithin{equation}{section}

\newcommand{\mf}{\mathbf}

\newcommand{\R}{\mathbb{R}}

\title[Codimension Bound]{Codimension Bounds and Rigidity of Ancient Mean Curvature Flows by the Tangent Flow at $-\infty$}

\author{Douglas Stryker}
\address{Department of Mathematics, Massachusetts Institute of Technology, 77 Massachusetts Avenue, Cambridge, MA 02139, USA}
\address{Current address: Department of Mathematics
	Princeton University
	Fine Hall, Washington Road
	Princeton, NJ 08544, USA}
\email{dstryker@princeton.edu}

\author{Ao Sun}
\address{Department of Mathematics, Massachusetts Institute of Technology, 77 Massachusetts Avenue, Cambridge, MA 02139, USA}
\address{Current address: Department of Mathematics,
University of Chicago,
5734 S. University Avenue,
Chicago, IL, 60637, USA}
\email{aosun@uchicago.edu}

\begin{document}

\maketitle

\begin{abstract}
Motivated by the limiting behavior of an explicit class of compact ancient curve shortening flows, by adapting the work of Colding-Minicozzi in \cite{CM-complexity}, we prove codimension bounds for ancient mean curvature flows by their tangent flow at $-\infty$. In the case of the $m$-covered circle, we apply this bound to prove a strong rigidity theorem. Furthermore, we extend this paradigm by showing that under the assumption of sufficiently rapid convergence, a compact ancient mean curvature flow is identical to its tangent flow at $-\infty$.
\end{abstract}

\section{Introduction}
A family of immersed $n$-dimensional submanifolds $M_t^n \subset \R^N$ evolves along the mean curvature flow if its coordinates satisfy the equation
\[
\partial_t x = -\vec{H},
\]
where $\vec{H}$ is the mean curvature vector, given by minus the trace of the second fundamental form. This equation can also be written as a geometric heat equation in the form
\[
\partial_t x = \Delta_{M_t} x.
\]
The mean curvature flow is the negative gradient flow for the volume of submanifolds induced by $\R^N$, so solutions to this flow optimally decrease their volume.
In particular, the one dimensional mean curvature flow, called the curve shortening flow, optimally decreases the length of immersed curves.

The mean curvature flow has been the subject of extensive study in codimension one (see \cite{bwhite02}, \cite{CM-generic}, \cite{CMP-MCF}, \cite{CM-uniqueness}); namely for hypersurfaces $M_t^n \subset \R^{n+1}$. The mean curvature flow for higher codimension in Euclidean space, which is the focus of this paper, presents a unique challenge (see \cite{Wang-Long-time}, \cite{Smoczyk}, \cite{CM-complexity}, \cite[\S 0.3]{CM-regularity}). For example, the avoidance principle used to handle flows with codimension one fails for higher codimension. To clarify, by codimension, we mean the codimension of the solution in the Euclidean subspace of minimal dimension that still contains the solution.

In this paper, the primary objects of study are \emph{ancient} solutions to the mean curvature flow. A solution is called ancient if it is defined for all time in the interval $(-\infty, 0)$. Ancient solutions to the mean curvature flow are models for the singularities formed under the flow. Therefore, the study of ancient mean curvature flows is important to the study of the singular behavior of the mean curvature flow. Ancient solutions in codimension $1$ have been extensively studied (see \cite{WangXuJia}, \cite{Huisken-Sinestrari-ancient}, \cite{Haslhofer-Hershkovits}, \cite{ADS}, \cite{Brendle-Choi}, \cite{choi2018ancient}).

Our first result is the following sharp codimension bound.

\begin{theorem}\label{thm:samecodim}
Let $\Sigma$ be a compact shrinker (i.e.~satisfying $\vec{H} \equiv \frac{x^{\perp}}{2}$), and let $\lambda_1$ be the smallest nonzero eigenvalue of the drift Laplacian $\mathcal{L}_{\Sigma} := \Delta_{\Sigma} + \frac{1}{2}\nabla_{x^T}$. Let $M_t^n \subset \R^N$ be an ancient mean curvature flow so that $\frac{M_t}{\sqrt{-t}}$ is an $\epsilon(t)$ $C^1$-graph over $\Sigma$, where
\[ \lim_{t \to -\infty} (-t)^{\frac{1}{2} - \delta}\epsilon(t) = 0
\]
for some $\delta < \lambda_1$. Then $\mathrm{codim}(M_t) = \mathrm{codim}(\Sigma)$. Moreover, $M_t$ lies in the same subspace of $\R^N$ as $\Sigma$ for all $t$.
\end{theorem}

Theorem \ref{thm:samecodim} adapts the work by Colding-Minicozzi in \cite{CM-complexity} to handle arbitrary compact shrinkers. In \cite{CM-complexity} Colding-Minicozzi bounded the complexity of ancient solutions to the mean curvature flow. In \cite[Theorem 0.15]{CM-complexity}, they proved a sharp codimension bound for ancient solutions whose tangent flows at $-\infty$ are round cylinders. We prove Theorem \ref{thm:samecodim} by adapting Colding-Minicozzi's techniques; however we need to overcome new difficulties that arise for arbitrary compact shrinkers instead of round cylinders.

As an example of the practical applications of Theorem \ref{thm:samecodim}, we apply our codimension bound to prove the following rigidity result.

\begin{corollary}\label{cor:circlerigidity}
Let $M_t \subset \R^N$ be an ancient curve shortening flow with only type I singularities so that $\frac{M_t}{\sqrt{-t}}$ is a $\epsilon(t)$ $C^1$-graph over the multiplicity $m$ circle $mS^1(\sqrt{2})$ with
\[ \lim_{t\to -\infty} (-t)^{\frac{1}{2} - \frac{1}{2m^2} + \rho}\epsilon(t) = 0\]
for some $\rho > 0$. Then $M_t \equiv mS^1(\sqrt{-2t})$.
\end{corollary}

Finally, we prove a general rigidity theorem for compact rescaled ancient mean curvature flows. For the simplicity of exposition, we prove the following theorem for codimension $1$ solutions. Recall that if $M_t$ is an ancient mean curvature flow defined for $t\in(-\infty,0)$, then we can define a rescaled mean curvature flow 
\[\widetilde{M_t}=\frac{M_{-e^{-t}}}{\sqrt{e^{-t}}}.\]

\begin{theorem}\label{thm:main rigidity-intro}
There exists $\alpha>0$ depending on the shrinker $\Sigma$ such that the following holds. Suppose $\widetilde{M_t}$ is a rescaled mean curvature flow, and can be written as a graph over $\Sigma$ when $-t$ is sufficiently large, i.e. \[\widetilde{M_t}=\{x\in\Sigma:x+\varphi(x,t)\mf{n}(x)\}\]
where $\varphi$ is a $C^2$ function with $\|\varphi\|_{C^2}\to 0$ as $t\to-\infty$.. If
\[\limsup_{t\to-\infty}\varphi^2 e^{-\alpha  t}=0,\]
then $\varphi\equiv 0$ on $\Sigma\times(-\infty,0)$. In other words, $\widetilde{M_t}\equiv \Sigma$.
\end{theorem}
This theorem shows that if the convergence of the ancient mean curvature flow (after rescaling) to its tangent flow at $-\infty$ is sufficiently fast (in particular, much faster than the rate in Theorem \ref{thm:samecodim}), then the ancient rescaled mean curvature flow must be identical with its tangent flow. This kind of rigidity result has appeared in many different contexts for the mean curvature flow (see \cite{Wang2014}, \cite{Wang2016}). Note that if we rescale back to the original mean curvature flow, we obtain the following corollary.

\begin{corollary}
There exists $\alpha>0$ depending on the shrinker $\Sigma$ such that the following holds. Suppose $M_t$ is a mean curvature flow, and $\frac{M_t}{\sqrt{-t}}$ can be written as a graph over $\Sigma$ when $-t$ is sufficiently large, i.e. \[\frac{M_t}{\sqrt{-t}}=\{x\in\Sigma:x+\varphi(x,t)\mf{n}(x)\}\]
where $\varphi$ is a $C^2$ function with $\|\varphi\|_{C^2}\to 0$ as $t\to-\infty$.. If
\[\limsup_{t\to-\infty}\varphi^2 (-t)^{\alpha}=0,\]
then $\varphi\equiv 0$ on $\Sigma\times(-\infty,0)$. In other words, $\frac{M_t}{\sqrt{-t}}\equiv \Sigma$.
\end{corollary}

\subsection{Higher Codimensional Mean Curvature Flow}

The relative scarcity of explicit examples of higher codimensional ancient mean curvature flows is one of the main challenges of this subfield. Presently, we know very few constructions for higher codimensional ancient mean curvature flows. Most of them are solitons of the Lagrangian mean curvature flow (see \cite{Lee-Wang}, \cite{Castro-Lerma}). Choi-Mantoulidis constructed ancient mean curvature flows from unstable minimal submanifolds in \cite{ChMa}.

In this paper, our results are inspired by the behavior of a particular class of ancient solutions to the curve shortening flow in high codimension, which we call torus curves. After the first draft of this paper, we became aware of an existing proof of this construction in \cite{altschuler2013zoo}. The construction in \cite{altschuler2013zoo} is motivated by the symmetry of $\R^n$, whereas here we concentrate on the implications of the solution's high codimensional properties. We mention two implications of this solution.

First, the torus curve solution suggests the sharp value of a codimension bound of Colding-Minicozzi \cite{CM-complexity}. Recall that the entropy of a submanifold $M \subset \R^N$ is defined as
\[ \lambda(M) := \sup_{s \in \R_{>0},~y \in \R^N} (4\pi)^{-\frac{n}{2}}\int_{sM+y} e^{-\frac{|x|^2}{4}}, \]
i.e.~the supremum of the Gaussian integral over all dilations and translations of the submanifold. See \cite{CM-generic} for further discussion. Colding-Minicozzi proved that an ancient mean curvature flow $M_t^n\subset \R^N$ must lie in a Euclidean subspace of dimension $d\leq C_n\sup_{t}\lambda(M_t)$ (see \cite[Corollary 0.6]{CM-complexity}). The torus curve suggests that the sharp value of $C_1$ should be $2/\lambda(S^1)$. See the discussion of Conjecture \ref{conj:c1}.

Second, the torus curve solution illustrates an interesting relation between the codimension and the tangent flow. Recall that the tangent flow of an ancient mean curvature flow is a self-shrinking mean curvature flow generated by the limit of $\frac{M_t}{\sqrt{-t}}$ as $t\to 0$ or $t\to -\infty$. Since the tangent flows are all self-shrinking and determined by the self-shrinker generated the flow, we use the self-shrinker to denote the tangent flow. For the torus curve solution, the tangent flow at $0$ is the embedded circle, and the tangent flow at $-\infty$ is the circle with multiplicity. Since the torus curve solution can have arbitrary codimension, the tangent flow at $0$ \emph{cannot} bound the codimension of an ancient solution. However, the torus curve solution does suggest that the tangent flow at $-\infty$ \emph{can} be used to bound the codimension, inspiring the rest of the results in the paper.

We remark that our codimension bounds are \emph{not} immediate consequences of the entropy-based codimension bound of Colding-Minicozzi. Since the constant $C_n$ is universal for all ancient solutions, their bound could be weak for some ancient solutions whose tangent flow at $-\infty$ is a compact shrinker. Moreover, since the value of $C_n$ is unknown, our result provides a nontrivial bound.

\subsection{Strategy of the Proof of the Codimension Bounds}
For a mean curvature flow $M_t$, the basic idea to bound the codimension coming from \cite{CM-complexity} is to bound the dimension of the space of caloric functions (i.e.~functions satisfying $\partial_t u = \Delta_{M_t} u$) with bounded polynomial growth in space and time. In particular, we study the dimension of the space
\[ \mathcal{P}_d(M_t) := \{u \text{~ancient}\mid \partial_t u = \Delta_{M_t} u \text{~and~} |u(x,t)| \leq C_u(1+|x|^d+|t|^{d/2})\}. \]
The reason such a dimension bound is useful is because the space $\mathcal{P}_1(M_t)$ contains the constant functions and the coordinate functions of $M_t$. Hence, the number of linearly independent coordinate function is at most $\mathrm{dim}~\mathcal{P}_1(M_t) - 1$. Moreover, if we control the rate of convergence of a compact solution to its limiting shrinker as $t \to -\infty$, we can bound the codimension of the solution by bounding the dimension of the space $\mathcal{P}_d(M_t)$ for some $d < 1$.

To prove a dimension bound for $\mathcal{P}_d(M_t)$, we argue by contradiction, assuming the space has a basis with one too many functions. The contradiction arises from the combination of a lower bound and an upper bound for the norms of these caloric functions, where the function norm is given by integration against the Gaussian.

To bound the norms of the basis functions from above, we assume the existence of caloric functions $\psi_i$ that are close to the eigenfunctions of the drift Laplacian $\mathcal{L}_{\Sigma}$ on the limiting shrinker $\Sigma$. Since there is an extra function in the basis for $\mathcal{P}_d(M_t)$, we can make a change of basis so that the last basis function is orthogonal to $1$ and $\psi_i$ for all $i$ with respect to integration against the Gaussian. Since $\frac{M_t}{\sqrt{-t}}$ is close to $\Sigma$, we can transplant this setup to $\Sigma$. The transplantation of the last basis function is then nearly orthogonal to the first few eigenfunctions of $\mathcal{L}_{\Sigma}$. Recall that if a function $u$ is orthogonal to the first $k$ eigenfunctions of $\mathcal{L}_{\Sigma}$, then Rayleigh's inequality gives
\[ 
\lambda_{k+1}\int_{\Sigma} u^2e^{-\frac{|x|^2}{4}} \leq \int_{\Sigma} |\nabla u|^2e^{-\frac{|x|^2}{4}}.
\]
We use this observation to bound the function norms from above.

Compared to the cylinder case from \cite[Theorem 0.15]{CM-complexity}, the precise argument to obtain the upper bound described above in our setting is more delicate. First, in general, there can be eigenvalues in the range $(0, 1/2)$, corresponding to non-coordinate eigenfunctions, which is not the case on round cylinders. Moreover, since these eigenfunctions are not coordinate functions, it is more challenging to bound the error coming from the transplantation of the setup to $\Sigma$ and back.

To bound the norms of the basis functions from below, we use a result established by Colding-Minicozzi that holds for any mean curvature flow (see \cite[Lemma 3.9]{CM-complexity}).
Under the assumption of an extra basis function, these bounds yield a contradiction.

For the sake of exposition, we deal with the compact case. While more technically involved, we suspect the same tools should work in the noncompact setting (see \cite[\S 7]{CM-complexity}).

\subsection{Organization of the Paper}
In \S2, we discuss the higher codimensional properties of the torus curve mean curvature flow.
In \S3, we prove the main codimension bounds.
Finally, in \S4, we prove the rigidity of ancient mean curvature flows under sufficiently rapid convergence as $t \to -\infty$.

\subsection*{Acknowledgements} Both authors are grateful to Professor William Minicozzi for his advisory and helpful suggestions and comments. We want to thank Christos Mantoulidis for bringing our attention to \cite{ChMa}. We also want to thank Professor David Jerison, Professor Ankur Moitra, and Dr.~Slava Gerovich for supporting our research. We also thank the anonymous referee for the helpful comments, especially further clarification of the decay rates between mean curvature flows and rescaled mean curvature flows.

\section{Torus curves}
The results of this paper are motivated by the behavior of the following ancient curve shortening flow, which is due to \cite[p.~1195-1198]{altschuler2013zoo}. For clarity, we write down the explicit construction here.

Let $k_1,\ \hdots,\ k_m$ be an increasing list of positive integers. We construct a $t$-parametrized family of curves $\gamma^{(k_1, \hdots, k_m)}_t \subset \R^{2m}$ (we denote it by $\gamma_t$ when the integers $k_j$ are implied for ease of notation) with coordinate functions of the form
\begin{equation}\label{eq:toruscurve}
(\gamma_t(\theta))_{2j-1} = r(t)^{k_j^2}\cos(k_j\theta)~,~(\gamma_t(\theta))_{2j} = r(t)^{k_j^2}\sin(k_j\theta)
\end{equation}
for $j = 1, \hdots, m$ and $\theta \in [0, 2\pi)$, where $r(t)$ is a positive function.

Intuitively, $\gamma_t$ is a curve on the torus \[  S^1(r^{k_1^2}) \times \hdots \times  S^1(r^{k_m^2}) \subset \R^{2m} \]
that wraps around the $j$th copy of the circle $k_j$ times at constant speed. 

We remark that in \cite{altschuler2013zoo}, $k_1,\cdots,k_m$ are only assumed to be a nondecreasing list, but then $\gamma_t$ may lie in a smaller dimensional Euclidean space. Therefore we require $k_1,\cdots,k_m$ be an increasing list of positive integers.

By standard ODE techniques, \cite{altschuler2013zoo} proved the following result.

\begin{proposition}[(14),(15) of \cite{altschuler2013zoo}]\label{prop:toruscurve}
There is a unique positive function $r(t)$ for $t\in(-\infty, 0)$ satisfying
\begin{equation}\label{eq:toruslimit}
\lim_{t \to -\infty} r(t) = +\infty \text{~and~} \lim_{t \to 0} r(t) = 0
\end{equation}
so that the family of curves $\gamma_t$ with coordinate functions given by (\ref{eq:toruscurve}) defines an ancient solution to the curve shortening flow. In particular, $\gamma_t$ does not lie in any $(2m-1)$-dimensional Euclidean subspace.
\end{proposition}

By studying the limiting behavior of this solution as $t \to -\infty$ and $t \to 0$, \cite{altschuler2013zoo} proved the following proposition.

\begin{proposition}[Examples after (16) of \cite{altschuler2013zoo}]\label{prop:torustangentflow}
The tangent flow to the solution $\gamma_t^{(k_1,\hdots,k_m)}$ at $t = -\infty$ is the multiplicity $k_m$ circle, and the tangent flow at $t = 0$ is the multiplicity $k_1$ circle.
\end{proposition}

Using the behavior of the solution from Proposition \ref{prop:torustangentflow}, we can compute the entropy of the solution.

\begin{corollary}\label{cor:torusentropy}
The curve shortening flow $\gamma_t^{(k_1,\hdots,k_m)}$ satisfies $\sup_t\lambda(\gamma_t) = k_m\lambda(S^1)$.
\end{corollary}
\begin{proof}
First, we bound $\sup_t \lambda(\gamma_t)$ from below. We compute
\[
\frac{1}{\sqrt{4\pi}} \int_{s\gamma_t} e^{-\frac{|x|^2}{4}} = s\sqrt{\pi} \exp\left(-\frac{s^2}{4}\sum_{j=1}^m r^{2k_j^2}\right)\sqrt{\sum_{j=1}^m k_j^2r^{2k_j^2}}.
\]
Setting $\tilde{s} = \sqrt{2}\left(\sum_{j=1}^m r^{2k_j^2}\right)^{-1/2}$, and recalling that $\lim_{t \to -\infty} r(t) = \infty$, we have
\[ 
\sup_t \lambda(\gamma_t) \geq \lim_{t \to -\infty} \frac{1}{\sqrt{4\pi}} \int_{\tilde{s}\gamma_t} e^{-\frac{|x|^2}{4}} = \lambda(S^1) \lim_{t \to -\infty} \frac{\sqrt{\sum_{j=1}^m k_j^2r^{2k_j^2}}}{\sqrt{\sum_{j=1}^m r^{2k_j^2}}} = k_m\lambda(S^1).
\]

Second, we bound $\sup_t \lambda(\gamma_t)$ from above. Let $\epsilon > 0$. By Proposition \ref{prop:toruscurve}, there is a $T_{\epsilon} < 0$ so that $r^{-1} < \epsilon$ and $r \geq 1$ for all $t \leq T_{\epsilon}$. We compute
\[ \begin{split}
\frac{1}{\sqrt{4\pi}}\int_{sr^{-k_m^2}\gamma_t + y} e^{-\frac{|x|^2}{4}}
& \leq \sqrt{k_m^2 + C\epsilon^2}\frac{s}{\sqrt{4\pi}}
\int_0^{2\pi}e^{-\frac{s^2}{4}((y_{2m-1}-\cos(k_m\theta))^2 + (y_{2m}-\sin(k_m\theta))^2)}d\theta\\
& = \sqrt{k_m^2 + C\epsilon^2}\frac{s}{\sqrt{4\pi}}
\int_0^{2\pi}e^{-\frac{s^2}{4}((y_{2m-1}-\cos(\theta))^2 + (y_{2m}-\sin(\theta))^2)}d\theta\\
& \leq \sqrt{k_m^2 + C\epsilon^2}\lambda(S^1),
\end{split} \]
where the second line follows from the change of variables $k_m\theta \mapsto \theta$ and the periodicity of the sinusoidal functions. By the monotonicity of entropy, we conclude that
\[ \sup_t \lambda(\gamma_t) = \lim_{t \to -\infty} \lambda(\gamma_t) \leq \lim_{\epsilon \to 0} \sqrt{k_m^2 + C\epsilon^2}\lambda(S^1) = k_m \lambda(S^1).
\]
Hence, we obtain the desired equality.
\end{proof}

Recall that in \cite[Corollary 0.6]{CM-complexity}, Colding-Minicozzi showed that there are universal constants $C_n$, depending only on the intrinsic dimension $n$, so that if $M_t^n \subset \R^N$ is an ancient solution to the mean curvature flow, then $M_t$ lies in a Euclidean subspace of dimension at most $C_n\sup_t \lambda(M_t)$. As an initial application of the existence of the ancient solution $\gamma_t$, we use the entropy computation in Corollary \ref{cor:torusentropy} to bound the constant $C_1$ in this result. For the torus curve solution $\gamma_t^{(k_1, \hdots, k_m)}$, we obtain the bound $\lambda(S^1)C_1 \geq \frac{2m}{k_m}$, which is maximized when $k_m$ is as small as possible. Since $k_1,\ \hdots,\ k_m$ must be an increasing list of positive integers, we have $k_m \geq m$. Hence, the torus curve solution gives the bound $\lambda(S^1)C_1 \geq 2$, identical to the bound from the shrinking circle solution in $\R^2$.

Since the constant speed sinusoidal functions are natural choices for linearly independent functions with compact images, this example suggests the following conjecture for the sharp constant $C_1$ in \cite[Corollary 0.6]{CM-complexity}.

\begin{conjecture}\label{conj:c1}
The sharp value of $C_1$ is $\frac{2}{\lambda(S^1)}$. In particular, any ancient curve shortening flow $M_t^1 \in \R^N$ that does not lie in a lower dimensional Euclidean subspace satisfies
\[ \sup_t \lambda(M_t) \geq \frac{N}{2}\lambda(S^1). \]
\end{conjecture}

As a second application of the existence of the ancient solution $\gamma_t$, we note that the codimension of an ancient solution cannot be bounded by information about its tangent flow as $t\to0$.

\begin{corollary}\label{cor:nobound}
For any integer $m$, there is an ancient curve shortening flow that does not lie in any $(2m-1)$-dimensional Euclidean subspace whose tangent flow as $t\to0$ is the embedded circle.
\end{corollary}
\begin{proof}
Take the torus curve solution $\gamma^{(k_1, \hdots, k_m)}$ with $k_1 = 1$ and apply Proposition \ref{prop:torustangentflow}.
\end{proof}

Despite the fact that information about an ancient solution as $t \to 0$ \emph{cannot} be used to bound its codimension, there is hope that information about the solution as $t \to -\infty$ \emph{can} bound the codimension. In the next section, motivated by the limiting behavior of the torus curve solution as $t \to -\infty$, we prove codimension bounds for ancient solutions that converge sufficiently rapidly to their tangent flow as $t \to -\infty$.

\section{Codimension bounds by the tangent flow at $-\infty$}
In this section, we prove sharp codimension bounds for ancient mean curvature flows using their limiting behavior as $t \to -\infty$. Recall that in \cite[\S 7]{CM-complexity}, Colding-Minicozzi proved a sharp codimension bound for ancient solutions whose tangent flow at $-\infty$ is a round cylinder. Here, we adapt the techniques of Colding-Minicozzi to handle ancient solutions whose tangent flow at $-\infty$ is an arbitrary \emph{compact} shrinker $\Sigma^n$.

The case of general shrinkers is more delicate than the case of round cylinders. This difficulty arises from the fact that, unlike for cylinders, the lowest nonzero eigenvalue of the drift Laplacian operator $\mathcal{L}_{\Sigma} := \Delta_{\Sigma} - \frac{1}{2}\nabla_{x^T}$ is in general less than $1/2$, corresponding to eigenfunctions that are not coordinate functions of the shrinker. For example, the spectrum of the drift Laplacian of the multiplicity $m$ circle is $\{\frac{k^2}{2m^2}\}_{k \geq 1}$, where each eigenvalue has multiplicity 2. Our key insight to handle this harder setting is to look at flows converging to the tangent flow at $-\infty$ at a rate related to the first nonzero eigenvalue of the drift Laplacian of $\Sigma$.

We recall the inner product notation used by Colding-Minicozzi (see \cite[\S 3]{CM-complexity}). Suppose $M_t^n$ is a mean curvature flow. If $u$ and $v$ are functions on $M_t$, we write
\[ J_t(u, v) := (-4\pi t)^{-\frac{n}{2}}\int_{M_t} uve^{\frac{|x|^2}{4t}} \text{~and~} I_u(t) := J_t(u, u). \]
Namely, $J_t$ is the inner product given by integration against the normalized Gaussian on $M_t$, and $I_u(t)$ is the squared norm of $u$ under this inner product. We freely use the fact that for a caloric function $u$ on a mean curvature flow, the function $I_u(t)$ is monotone non-increasing in time (see \cite[Lemma 3.4]{CM-complexity}). We also recall that $\mathcal{L}_{\Sigma}$ is a self-adjoint operator in the space of Gaussian weighted functions. Therefore many classical properties of self-adjoint operators are valid for $\mathcal{L}_{\Sigma}$; for instance, the existence of eigenfunctions. We refers the reader to \cite[~Section 3]{CM-generic} for further discussion of $\mathcal{L}_{\Sigma}$.

Let $\Sigma^n$ be a compact shrinker. Let $\{\lambda_i\}_{i\geq 1}$ denote the nonzero eigenvalues of the drift Laplacian $\mathcal{L}_{\Sigma} := \Delta_{\Sigma} - \frac{1}{2}\nabla_{x^T}$ in non-decreasing order (counted with multiplicity), and let $\{\overline{\phi}_i\}_{i \geq 1}$ be an orthonormal collection of corresponding eigenfunctions. We use the bar notation in $\overline{\phi}_i$ to denote that the function is on $\Sigma$.

\subsection{Codimension rigidity}
In this subsection, we show that if a rescaled compact ancient solution converges sufficiently rapidly to its limit shrinker as $t \to -\infty$, then the solution lies in the same subspace as the limit shrinker.

For a fixed $t < 0$, let $M^n \subset \R^N$ be a submanifold so that $\frac{M}{\sqrt{-t}}$ is a graph over $\Sigma$. Then a function $u$ on $M$ can be transplanted to a function $\overline{u}$ on $\Sigma$, and vice versa. In particular, we let $\phi_i$ denote the transplantation of the eigenfunction $\overline{\phi}_i$ on $\Sigma$ to $M$.

As outlined in the introduction, we proceed by bounding the Gaussian norm of caloric functions from above and below to obtain a contradiction.

For the upper bound, we begin by adapting the Poincar\'e inequality from \cite[Lemma 7.14]{CM-complexity} to the setting of an arbitrary compact shrinker.

\begin{lemma}\label{lem:poincare}
Given $t < 0$ and $\mu > 0$, there is an $\epsilon_{\mu} > 0$ so that the following holds. Let $M^n \subset \R^N$ be a compact immersed submanifold with $\lambda(M) \leq \lambda_0 < \infty$, so that $\frac{M}{\sqrt{-t}}$ is an $\epsilon_{\mu}$ $C^1$-graph over the compact shrinker $\Sigma$. If $u$ satisfies $\int_{M} ue^{\frac{|x|^2}{4t}} = 0$, then
\begin{equation}\label{eq:poincare1} (1-\mu)(-4\pi t)^{-\frac{n}{2}}\int_{M} u^2e^{\frac{|x|^2}{4t}} \leq \frac{-t}{\lambda_1}(-4\pi t)^{-\frac{n}{2}}\int_{M} |\nabla u|^2e^{\frac{|x|^2}{4t}}. \end{equation}
\end{lemma}
\begin{proof}
For ease of notation, we write $\fint_{\Sigma} \overline{v} := (4\pi)^{-\frac{n}{2}}\int_{\Sigma} \overline{v}e^{-\frac{|x|^2}{4}}$ for a function $\overline{v}$ on $\Sigma$. We also write $J_t(u,v)$ and $I_u(t)$ to denote the time $t$ inner product and squared norm on $M$.

The crux of the proof is the well-known Rayleigh inequality, which says that if a function $\overline{v}$ on $\Sigma$ is orthogonal to the first $l$ eigenfunctions of $\mathcal{L}_{\Sigma}$, then
\begin{equation}\label{eq:rayleighineq} \fint_{\Sigma} \overline{v}^2 \leq \frac{1}{\lambda_{l+1}}\fint_{\Sigma} |\nabla \overline{v}|^2.
\end{equation}

First, we state the relevant consequences of the $C^1$-closeness of $\frac{M}{\sqrt{-t}}$ and $\Sigma$. By straightforward calculations similar to \cite[(7.21-7.23)]{CM-complexity}, we have
\begin{equation}\label{eq:transplantnorm}
\left|I_u(t) - \fint_{\Sigma}\overline{u}^2 \right| \leq O(\epsilon_{\mu})I_u(t),
\end{equation}
\begin{equation}\label{eq:transplantip1}
\left(\fint_{\Sigma} \overline{u} \right)^2 = \left|J_t(u, 1) - \fint_{\Sigma} \overline{u} \right|^2 \leq O(\epsilon_{\mu}^2)\lambda_0I_u(t),
\end{equation}
\begin{equation}\label{eq:transplantenergy}
\left|I_{|\nabla u|}(t) - \frac{1}{-t}\fint_{\Sigma} |\nabla \overline{u}|^2 \right| \leq O(\epsilon_{\mu})I_{|\nabla u|}(t),
\end{equation}
where we recall that $\overline{u}$ denotes the transplantation of $u$ to $\Sigma$.

Since $\overline{u}$ is not exactly orthogonal to $1$, we have to orthogonally project before applying Rayleigh's inequality (\ref{eq:rayleighineq}). We obtain
\[
\fint_{\Sigma} \left(\overline{u} - \frac{\fint_{\Sigma} \overline{u}}{\lambda(\Sigma)}\right)^2 \leq \frac{1}{\lambda_{1}}\fint_{\Sigma} \left|\nabla \overline{u}\right|^2.
\]
For the left hand side, we compute (using the fact that $\lambda(\Sigma) \geq 1$)
\[ \begin{split}
\fint_{\Sigma} \left(\overline{u} - \frac{\fint_{\Sigma} \overline{u}}{\lambda(\Sigma)}\right)^2
& = \fint_{\Sigma} \overline{u}^2 -\frac{\left(\fint_{\Sigma} \overline{u}\right)^2}{\lambda(\Sigma)} \geq \fint_{\Sigma} \overline{u}^2 -\left(\fint_{\Sigma} \overline{u}\right)^2.
\end{split} \]
Hence, we obtain
\begin{equation}\label{eq:rayleigh2}
\fint_{\Sigma} \overline{u}^2 \leq \left(\fint_{\Sigma} \overline{u}\right)^2 + \frac{1}{\lambda_{1}}\fint_{\Sigma} |\nabla \overline{u}|^2.
\end{equation}

Combining (\ref{eq:transplantnorm})-(\ref{eq:transplantenergy}) with (\ref{eq:rayleigh2}), we obtain
\[ \begin{split}
(1 - O(\epsilon_{\mu}))I_u(t)
&\leq \fint_{\Sigma} \overline{u}^2
\leq \left(\fint_{\Sigma} \overline{u}\right)^2 + \frac{1}{\lambda_{1}}\fint_{\Sigma} |\nabla \overline{u}|^2 \leq O(\epsilon_{\mu}^2)I_u(t) + \frac{-t}{\lambda_{1}}(1+O(\epsilon_{\mu}))I_{|\nabla u|}(t).
\end{split} \]
Taking $\epsilon_{\mu}$ sufficiently small, we obtain (\ref{eq:poincare1}).
\end{proof}

Using Lemma \ref{lem:poincare}, we adapt the caloric functions norm upper bound \cite[Lemma 7.30]{CM-complexity} to our setting.

\begin{lemma}\label{lem:upperbound1}
Let $M_t^n\subset \R^N$ be a mean curvature flow defined for $t_1 \leq t \leq t_2 < 0$ with $\lambda(M_t) \leq \lambda_0 < \infty$. Given $\mu \in (0, 1/2)$, there is an $\epsilon_{\mu} > 0$ independent of $t \in [t_1, t_2]$ so that if
\begin{itemize}
\item $\frac{M_t}{\sqrt{-t}}$ is an $\epsilon_{\mu}$ $C^1$-graph over the compact shrinker $\Sigma^n$ for $t \in [t_1, t_2]$,
\item $\partial_t u = \Delta_{M_t} u$, $I_u(t_1) = 1$, and $J_{t_1}(u, 1) = 0$,
\end{itemize}
then
\[ I_u(t_2) \leq \left(\frac{t_1}{t_2}\right)^{2\lambda_1(\mu-1)} + O(\epsilon_{\mu}). \]
\end{lemma}
\begin{proof}
We show that we have the differential inequality
\begin{equation}\label{eq:diffineq}
((-t)^{2\lambda_1(\mu-1)}I_u)' \leq O(\epsilon_{\mu})(-t)^{2\lambda_1(\mu-1)-1}.
\end{equation}
Integrating from $t_1$ to $t_2$ and using $I_u(t_1) = 1$, we get
\[ \begin{split} (-t_2)^{2\lambda_1(\mu-1)}I_u(t_2) \leq (-t_1)^{2\lambda_1(\mu-1)} + O(\epsilon_{\mu})(-t_2)^{2\lambda_1(\mu-1)}, \end{split} \]
which gives the desired inequality after dividing by $(-t_2)^{2\lambda_1(\mu-1)}$. Hence, it suffices to show (\ref{eq:diffineq}).

To apply Lemma \ref{lem:poincare} at time $t$, we need a function that is $J_t$-orthogonal to $1$. Hence, we define the function $v := u - \frac{J_t(u, 1)}{I_1(t)}$, which satisfies this property. 
Hence, by Lemma \ref{lem:poincare}, we have
\[ (1-\mu)I_v(t) \leq \frac{-t}{\lambda_1} I_{|\nabla v|}(t). \]
By the definition of $v$ we have
\[
I_u(t) = I_v(t) + \frac{J_t^2(u, 1)}{I_1(t)}~~~ \text{~and~}~~~ |\nabla v| = |\nabla u|.
\]
Therefore, by the weighted monotonicity formula (see \cite[Theorem 4.13]{ecker} and \cite[(3.6)]{CM-complexity}), we have
\[ (1-\mu)I_u(t) \leq \frac{t}{2\lambda_1}I_u'(t) + \frac{J_t^2(u, 1)}{I_1(t)}.
\]
By \cite[Lemma 7.1]{CM-complexity} and \cite[(7.23)]{CM-complexity}, the rightmost term is bounded by $O(\epsilon_{\mu})$. Multiplying by $2\lambda_1(-t)^{2\lambda_1(\mu-1)-1}$, we obtain (\ref{eq:diffineq}).
\end{proof}

Now, we bound caloric function norms from below. This part of the argument follows \cite[Lemma 3.9]{CM-complexity}.

Let $d$ be a nonnegative real number. Let $u_0 \equiv 1,\ u_1, \hdots, u_p \in \mathcal{P}_{d}(M_t)$ be linearly independent in spacetime. Using Gram-Schmidt, we orthogonalize these functions with respect to the inner products $J_t$.

Choose $t_0 < 0$. Let $w_{0, t_0} := u_0$. For $i = 1, \hdots, p$, we choose constants $\lambda_{j,i}(t_0) \in \R$ so that the function
\begin{equation}\label{eq:ortho} w_{i, t_0} := u_i - \sum_{j=0}^{i-1} \lambda_{j,i}(t_0) u_j
\end{equation}
is $J_{t_0}$-orthogonal to $u_0, \hdots, u_{i-1}$. Let $f_i(t_0):= I_{w_{i, t_0}}(t_0)$.

In this setup, Colding-Minicozzi proved the following caloric function norm lower bound.

\begin{lemma}{\cite[Lemma 3.9]{CM-complexity}}\label{lem:lowerbound}
Given $\mu > 0$, $\Omega > 1$, there is a sequence $m_q \to \infty$ so that the functions $v_i$ defined by
\[ v_i := w_{i,-\Omega^{m+1}}/\sqrt{f_i(-\Omega^{m+1})} \]
satisfy
\begin{equation}\label{eq:lem39} J_{-\Omega^{m+1}}(v_i, v_j) = \delta_{ij} \text{~and~} \sum_{i = 1}^p I_{v_i}(-\Omega^{m}) \geq p\Omega^{-\mu - d}. \end{equation}
\end{lemma}

In their proof of \cite[Lemma 3.9]{CM-complexity}, Colding-Minicozzi referenced prior work (\cite[Proposition 4.16]{cm97}) to establish that the functions $f_i$ are non-increasing and have growth bounded by $C_i(1-t)^d$. While the proof of \cite[Proposition 4.16]{cm97} was written for harmonic functions, the proof in the parabolic setting is formally identical.

Now, we are equipped with the tools to prove the rigidity of the codimension of ancient solutions.

\begin{proof}[Proof of Theorem \ref{thm:samecodim}]
Without loss of generality, suppose $\overline{x}_1, \hdots, \overline{x}_k$ are linearly independent and the only nonzero coordinate functions of $\Sigma$. By the convergence assumption, $M_t$ possesses coordinate functions $x_1, \hdots, x_k$ so that $\frac{x_i}{\sqrt{-t}}$ converges to $\overline{x}_i$ in the $C^1$ norm.

Moreover, the assumption of $C^1$ convergence at rate $\epsilon(t)$ implies that
\[ \Big|\frac{x_i}{\sqrt{-t}} - \overline{x_i}\Big| = \Big|\frac{x_i}{\sqrt{-t}}\Big| \leq \epsilon(t) \]
for all $i > k$. Then by the decay assumption for $\epsilon(t)$, we have $x_i \in \mathcal{P}_{2\delta}(M_t)$ for all $i > k$. Since the constant functions lie in $\mathcal{P}_{2\delta}(M_t)$, it suffices to show that $\mathrm{dim}~\mathcal{P}_{2\delta}(M_t) = 1$.
Suppose for the sake of contradiction that $u_0 \equiv 1,\ u_1 \in \mathcal{P}_{2\delta}(M_t)$ are linearly independent.

First, we apply the caloric function lower bound. Let $\Omega > 1$ and $\mu > 0$. Applying the orthonormalization procedure in (\ref{eq:ortho}) and Lemma \ref{lem:lowerbound}, there is a sequence $m_q \to \infty$ so that the functions $\{v_0 \equiv c,\ v_1\}$ satisfy
\begin{equation}\label{eq:lowerboundthm1}
J_{-\Omega^{m_q+1}}(v_i, v_j) = \delta_{ij} \text{~and~} I_{v_1}(-\Omega^{m_q}) \geq \Omega^{-2{\delta} - \mu}.
\end{equation}

Second, we apply the caloric function upper bound. By Lemma \ref{lem:upperbound1}, we have
\begin{equation}\label{eq:upperboundthm1}
I_{v_1}(-\Omega^{m_q}) \leq \Omega^{2\lambda_1(\mu - 1)} + C\epsilon(-\Omega^{m_q}).
\end{equation}

Together, (\ref{eq:lowerboundthm1}) and (\ref{eq:upperboundthm1}) imply
\begin{equation}\label{eq:ineqthm1}
\Omega^{-2{\delta} - \mu} \leq \Omega^{2\lambda_1(\mu - 1)} + C\epsilon(-\Omega^{m_q}).
\end{equation}
Since $\epsilon$ tends to $0$ as $t \to -\infty$, the second term on the right hand side is arbitrarily small for $q$ large. Since $\delta < \lambda_1$, we can take $\mu$ sufficiently small so that $\Omega^{-2{\delta} - \mu} > \Omega^{2\lambda_1(\mu-1)}$, which contradicts (\ref{eq:ineqthm1}).
\end{proof}

\begin{remark}
The torus curve solution indicates that the convergence assumption in Theorem \ref{thm:samecodim} is essentially sharp. For example, the torus curve solution $\gamma_t^{(1,2)} \subset \R^4$ satisfies the decay assumption $(-t)^{\frac{1}{2} - \lambda_1}\epsilon(t) \leq C < \infty$ and has higher codimension, where $\lambda_1$ of the limit shrinker $2S^1(\sqrt{2})$ is $\frac{1}{8}$.
\end{remark}

\subsection{Application: Rigidity of the circle with multiplicity}

By Theorem \ref{thm:samecodim}, if $M_t$ is an ancient curve shortening flow whose rescaled flow converges sufficiently quickly to the multiplicity $m$ circle $mS^1(\sqrt{2})$ as $t \to -\infty$, then $M_t$ is planar. Combining this fact with an entropy bound from \cite{baldauf}, we obtain the following rigidity result.

\begin{proof}[Proof of Corollary \ref{cor:circlerigidity}]
By the convergence assumption, there is a $T < 0$ so that $M_t$ has turning number $m$ for all $t \leq T$. By the type I singularity assumption, \cite[Theorem A]{baldauf} implies that $M_T$ has entropy $\lambda(M_T) \geq m\lambda(S^1)$. Since $\frac{M_t}{\sqrt{-t}}$ converges to $mS^1(\sqrt{2})$ as a $C^1$ graph as $t \to -\infty$, we have $\sup_t \lambda(M_t) = m\lambda(S^1)$. Then by the monotonicity of entropy, we have $\lambda(M_t) \equiv m\lambda(S^1)$ for all $t$, which implies that $M_t \equiv mS^1(\sqrt{-2t})$.
\end{proof}

\begin{remark}
The type I assumption is necessary. In fact, applying the construction in \cite[Theorem 1.6]{ChMa} to the Gaussian area functional in $\R^2$ gives a nontrivial rescaled ancient curve shortening flow $M_t/\sqrt{-t}$ converging to $mS^1(\sqrt{2})$ exponentially fast as $t \to -\infty$, where $M_t/\sqrt{-t}$ realizes an unstable perturbation of $mS^1(\sqrt{2})$ (we note that although their construction was developed for minimal surfaces in a closed manifold, it should not be too hard to generalize to closed minimal surfaces in a noncompact manifold).

Therefore, after undoing the rescaling, $M_t$ is not identical to $mS^1(\sqrt{-2t})$. However, by the sharp entropy bound in \cite{baldauf}, this solution has type II singularities.
\end{remark}

\section{Rigidity of ancient mean curvature flows}
The goal of this section is to prove a rigidity theorem for ancient mean curvature flows. Let $\Sigma$ be a closed self-shrinker with trivial normal bundle. We show that if the rescaled mean curvature flow $\widetilde{M_t}$ converges to $\Sigma$ sufficiently fast, then $\widetilde{M_t}$ is identical with $\Sigma$.

For the simplicity of exposition, we prove the rigidity theorem for codimension $1$ ancient mean curvature flows. The proof for higher codimensional ancient mean curvature flows is similar, but it needs more complicated computations on the normal bundle of the limit.

Recall that the evolution equation of the rescaled mean curvature flow is given by
\begin{equation}\label{Eq:RMCF}
\partial_t \tilde x=-(\tilde H-\frac{1}{2}\langle \tilde x,\tilde{\mf n}\rangle)\tilde{\mf n},
\end{equation}
see \cite{huisken}, \cite[(2.1)]{CIMW}. Here $\tilde x$ is the positions of the rescaled mean curvature flow, $\tilde{\mf n}$ is the normal vector of the rescaled mean curvature flow, $\tilde{H}$ is the mean curvature of the rescaled mean curvature flow.

We prove the following theorem. 
\begin{theorem}\label{thm:main rigidity}
There exists $\alpha>0$ depending on the shrinker $\Sigma$ such that the following holds. Suppose $\widetilde{M_t}$ can be written as a graph over $\Sigma$ when $-t$ is sufficiently large, i.e. \[\frac{M_t}{\sqrt{-t}}=\{x\in\Sigma:x+\varphi(x,t)\mf{n}(x)\}\]
where $\varphi$ is a $C^2$ function with $\|\varphi\|_{C^2}\to 0$ as $t\to-\infty$. If
\[\limsup_{t\to-\infty}\varphi^2 e^{-\alpha  t}=0,\]
then $\varphi\equiv 0$ on $\Sigma\times(-\infty,0)$. In other words, $\frac{M_t}{\sqrt{-t}}\equiv \Sigma$.
\end{theorem}

The main tool is the following Carleman inequality for $C^2$ functions on $\Sigma\times(-\infty,0)$.

\begin{lemma}\label{lem:carlemann inequality}
Let $u$ be a $C^2$ function on $\Sigma\times(-\infty,0)$. Then for any $T_1< T_2<0$ and $\alpha>0$, $\delta>0$, we have
\[
\int_{T_1}^{T_2}\int_{\Sigma}\left((\alpha-\delta^{-1})u^2+2|\nabla u|^2\right)e^{-\alpha t}
\leq 
\int_{T_1}^{T_2}\int_{\Sigma}\delta(u_t-\Delta u)^2 e^{-\alpha t}
+\int_{\Sigma}u^2(\cdot,T_1)e^{-\alpha {T_1}}.
\]
\end{lemma}

\begin{proof}
Define $h(t)=e^{-\alpha t}$. Note that $h_t=-\alpha h$. We have the following computations.
\[
\begin{cases}
\partial_t(u^2h)=2u_t uh+u^2h_t=2u_t uh-\alpha u^2 h,\\
\nabla(u^2h)=2uh\nabla u,\\
\Delta (u^2h)= 2\Delta u uh+2|\nabla u|^2 h.
\end{cases}\]
Hence, we have
\[\partial_t(u^2 h)-\Delta(u^2 h)=2uh(u_t-\Delta u)-\alpha u^2 h-2|\nabla u|^2 h.\]
Integrating both sides on $\Sigma\times[T_1,T_2]$, we obtain
\[\int_{T_1}^{T_2}\int_\Sigma \partial_t(u^2h)-\Delta(u^2 h)=\int_{T_1}^{T_2}\int_\Sigma \left(2uh(u_t-\Delta u)-\alpha u^2 h-2|\nabla u|^2 h\right).\]
Integration by parts gives
\[\int_{\Sigma}u^2(\cdot,T_2) e^{-\alpha T_2}-\int_{\Sigma}u^2(\cdot,T_1) e^{-\alpha T_1}=\int_{T_1}^{T_2}\int_\Sigma 2u(u_2-\Delta u)e^{-\alpha t}-\int_{T_1}^{T_2}\int_\Sigma\left(\alpha u^2+2|\nabla u|^2\right)e^{-\alpha t}.\]
A basic absorbing inequality gives
\[2u(u_t-\Delta u)\leq \delta^{-1} u^2+\delta(u_t-\Delta u)^2.\]
Hence, we obtain
\[
\begin{split}
\int_{T_1}^{T_2}\int_{\Sigma}\left((\alpha-\delta^{-1})u^2+2|\nabla u|^2\right)e^{-\alpha t}
&\leq 
\int_{T_1}^{T_2}\int_{\Sigma}\delta(u_t-\Delta u)^2 e^{-\alpha t}
-\int_{\Sigma}u^2(\cdot,t)e^{-\alpha {t}}\Big|_{T_1}^{T_2}\\
&\leq \int_{T_1}^{T_2}\int_{\Sigma}\delta(u_t-\Delta u)^2 e^{-\alpha t}
+\int_{\Sigma}u^2(\cdot,T_1)e^{-\alpha {T_1}}.
\end{split}
\]
\end{proof}

\begin{proof}[Proof of Theorem \ref{thm:main rigidity}]
According to the evolution equation of the rescaled mean curvature flow, we have
\begin{equation}\label{Eq:Key inequality}
|\partial_t \varphi-\Delta\varphi|^2\leq C(|\varphi|^2+|\nabla \varphi|^2)
\end{equation}
for some constant $C$ when $t$ is sufficiently large. These computations have appeared in the proof of \cite[Lemma 2.4]{Wang2014}. For the reader's convenience, we include the computations, but postpone them to the end of this section. Then Lemma \ref{lem:carlemann inequality} implies that for any $\alpha>0$, $\delta>0$,
\[
\int_{T_1}^{T_2}\int_{\Sigma}\left((\alpha-\delta^{-1})\varphi^2+2|\nabla \varphi|^2\right)e^{-\alpha t}
\leq 
\int_{T_1}^{T_2}\int_{\Sigma}C\delta(|\varphi|^2+|\nabla \varphi|^2) e^{-\alpha t}
+\int_{\Sigma}\varphi^2(\cdot,T_1)e^{-\alpha {T_1}}.
\]
So if we choose $\delta$ sufficiently small such that $C\delta<1$ and $\alpha$ sufficiently large such that $\alpha-\delta^{-1}>2$, we have
\[\int_{T_1}^{T_2}\int_\Sigma \left(\varphi^2+|\nabla\varphi|^2 \right)e^{-\alpha t}\leq \int_{\Sigma}\varphi^2(\cdot,T_1)e^{-\alpha {T_1}}.\]
So if $\limsup_{t\to-\infty}\varphi^2 e^{-\alpha  t}=0$, we have
\[\limsup_{T_1\to -\infty}\int_{T_1}^{T_2}\int_\Sigma \left(\varphi^2+|\nabla\varphi|^2 \right)e^{-\alpha t}\leq 0\ \text{for any $T_2>0$}.\]
So $\varphi\equiv 0$ on $\Sigma\times(-\infty,0)$.
\end{proof}

In the end of this section, we show the inequality (\ref{Eq:Key inequality}).

\begin{proof}[Proof of (\ref{Eq:Key inequality})]
The proof is the same as the proof of \cite[Lemma 2.4]{Wang2014} with slightly modification. For the reader's convenience we provide the details here. Given $x_0\in\Sigma$, let us choose a local parametrization of $\Sigma$ in a neighbourhood of $x_0$, $F:\Omega\to\Sigma$, satisfying the following properties:
\begin{itemize}
\item $\Omega$ is a domain in $\R^n$ containing $0$ and $F(0)=x_0$;
\item $\langle\partial_i F(0),\partial_j F(0)\rangle=\delta_{ij}$;
\item $\partial_{ij}F(0)=a_{ij}(0)\mf n(x_0)$, here $a_{ij}=A(\partial_i F(0),\partial_j F(0))$ is the second fundamental form, and $a_{ij}(0)=0$ if $i\neq j$.
\end{itemize}
Let $\tilde{F}(p)=F(p)+\varphi(p)\mf n(p)$ be a parametrization of $\widetilde{M_t}$. Here and from now on, we identify $\varphi(p)$ and $\mf n(p)$ with $\varphi(F(p))$ and $\mf n(F(p))$ respectively. Now we can compute the geometric quantities on $\widetilde{M_t}$. 

First, the tangent vectors are given by
\[\partial_i\tilde{F}=\partial_i F+(\partial_i\varphi)\mf n+\varphi\partial_i \mf n.\]
Thus the normal vector of $\widetilde{M_t}$ at $\tilde{F}(0)$ is parallel to the following vector
\begin{equation}
\mf N=-\sum_k\left[\prod_{l\neq k}(1-a_{ll}\varphi)\right](\partial_k\varphi)\partial_k F
+
\left[\prod_k(1-a_{kk}\varphi)\right]\mf n.
\end{equation}
Furthermore, we can compute the second order derivatives of $\tilde{F}$:
\[\partial_{ij}\tilde{F}=\partial_{ij}F+(\partial_{ij}\varphi)\mf n+(\partial_i\varphi)\partial_j \mf n+(\partial_j\varphi)\partial_i\mf n+\varphi\partial_{ij}\mf n.\]
At $p=0$, we can compute 
\[\partial_{ij}\mf n
=\sum_k\langle\partial_k F,\partial_{ij}\mf n\rangle\partial_k F+\langle\partial_{ij}\mf n,\mf n\rangle\mf n
=-\sum_k(\partial_j a_{ik})\partial_k F-a_{ii}a_{jj}\delta_{ij}\mf n.\]
Thus at $p=0$ the second order derivatives of $\tilde{F}$ is given by
\[
\partial_{ij}\tilde{F}
=-a_{ii}(\partial_j \varphi)\partial_i F-a_{jj}(\partial_i\varphi)\partial_j F
-\sum_k(\partial_j a_{ik})\varphi\partial_k F
+(a_{ij}-a_{ii}a_{jj}\delta_{ij}+\partial_{ij}\varphi)\mf n.
\]
Furthermore,
\begin{equation}\label{Eq:2nd derivative}
\begin{split}
\langle\partial_{ij}\tilde{F},\mf N\rangle
=&
a_{ii}(\partial_i\varphi)(\partial_j\varphi)\prod_{k\neq i}(1-a_{kk}\varphi)
+a_{jj}(\partial_i\varphi)(\partial_j\varphi)\prod_{k\neq j}(1-a_{kk}\varphi)\\
+&
(a_{ij}-a_{ii}a_{jj}\delta_{ij}\varphi+\partial_{ij}\varphi)\prod_k(1-a_{kk}\varphi)
+\varphi\sum_k\left[\prod_{l\neq k}(1-a_{ll}\varphi)\right](\partial_j a_{ik})\partial_k\varphi.
\end{split}
\end{equation}

Next we compute the pullback metric $\tilde{g}_{ij}$ from $\widetilde{M_t}$ at $p=0$:
\[\tilde{g}_{ij}=\langle\partial_i\tilde{F},\partial_j\tilde{F}\rangle=(1-a_{ii}\varphi)(1-a_{jj}\varphi)\delta_{ij}+(\partial_i\varphi)(\partial_j\varphi).\]
So at $p=0$ the determinant and the inverse of the metric are given by
\[\det(\tilde g)=1-2\sum_k a_{kk}\varphi+Q_1(p,\varphi,D\varphi),\]
\[\tilde g^{ij}\det(\tilde g)=
\begin{cases}
Q_{2ij}(p,\varphi,D\varphi)\partial_i\varphi,\quad i\neq j,\\
1-2\sum_{k\neq i}a_{kk}\varphi+Q_{2ij}(p,\varphi,D\varphi), \quad i=j.
\end{cases}\]
Here $Q_1$ and $Q_{2ij}$ are polynomials in $\varphi(p)$ and $D\varphi(p)$. Note that we have assumed that when $t$ is sufficiently large, $\|\varphi\|_{C^2}$ is sufficiently small. Hence
\[|Q_1(p,\varphi,D\varphi)|\leq C(|\varphi|+|D\varphi|),
\quad
|Q_{2ij}(p,\varphi,D\varphi)|\leq C(|\varphi|+|D\varphi|).\]

We also have
\[
\langle\tilde{F},\mf N\rangle
=
-\sum_{k}\left[\prod_{l\neq k}(1-a_{ll}\varphi)\right]\langle F,\partial_k F\rangle\partial_k\varphi
+(\langle F,\mf n\rangle+\varphi)\prod_k (1-a_{kk}\varphi).
\]

Taking inner product of (\ref{Eq:RMCF}) with $\mf N$ gives at $p=0$
\begin{equation}\label{Eq:RMCFvarphi}
\langle\partial_t \tilde{F},\mf N\rangle
=-(\tilde{H}-\frac{1}{2}\langle\tilde{F},\tilde{\mf n}\rangle)\langle\tilde{\mf n},\mf N\rangle
=-\tilde{H}\langle\tilde{\mf n},\mf N\rangle+\frac{1}{2}\langle \tilde{F},\mf N\rangle.
\end{equation}
Plugging in all the identities above to (\ref{Eq:RMCFvarphi}) gives
\begin{equation}
\partial_t\varphi \langle \mf n,\mf N\rangle = \tilde{g}^{ij}\langle\partial_{ij}\tilde{F},\mf N\rangle+\frac{1}{2}\langle \tilde{F},\mf N\rangle.
\end{equation}
Dividing both sides by $\langle \mf n,\mf N\rangle$ gives
\begin{equation}
\partial_t \varphi =\Delta\varphi+Q(p,\varphi,D\varphi,D^2\varphi),
\end{equation}
and 
\[|Q(p,\varphi,D\varphi,D^2\varphi)|\leq C(|\varphi|+|D\varphi|).\]
Note in the last inequality, there is no $|D^2\varphi|$ term. We can see this from the computations above: any second order derivative terms of $\varphi$ appear in (\ref{Eq:2nd derivative}). When $i=j$, $\partial_{ii}\varphi$ times $1$ contributes to the Laplacian $\Delta\varphi$ term. When $i\neq j$, $\partial_{ij}\varphi$ multiplies with the $\tilde{g}^{ij}$ term, so they multiply the $Q_{2ij}\partial_i\varphi$ term. Hence, it is controlled by the $C|D\varphi|$ term. Therefore, all the second order terms are controlled by $C(|\varphi|+|D\varphi|)$. This concludes the desired inequality (\ref{Eq:Key inequality}).
\end{proof}

\bibliography{bibfile}
\bibliographystyle{alpha}

\end{document}